\documentclass{article}
\usepackage[margin = 1in]{geometry}
\usepackage[utf8]{inputenc}

\usepackage{graphicx}
\usepackage{amssymb,amsmath,mathtools,amscd}
\usepackage{amsthm}
\usepackage{mathabx}
\usepackage{epstopdf}
\usepackage{tikz}
\usepackage{indentfirst}
\usepackage{hyperref}
\hypersetup{
	colorlinks=true,
	linkcolor=black,
	filecolor=black,      
	urlcolor=black,
	citecolor=blue,
}

\newcommand{\CC}{\mathbb{C}}
\newcommand{\RR}{\mathbb{R}}
\newcommand{\QQ}{\mathbb{Q}}
\newcommand{\ZZ}{\mathbb{Z}}

\newcommand{\KK}{\mathbf{k}}
\newcommand{\LL}{\mathbf{K}}

\newcommand{\bfield}{\beta_{\mathrm{field}}}

\newcommand{\gfield}{\gamma_{\mathrm{field}}}

\newcommand{\bfx}{\mathbf{x}}
\newcommand{\bfa}{\mathbf{a}}

\newcommand{\bfe}{\mathbf{e}}

\newcommand{\Supp}{\operatorname{Supp^\prime}}

\newtheorem{theorem}{Theorem}[section]
\newtheorem{lemma}[theorem]{Lemma}
\newtheorem{corollary}[theorem]{Corollary}
\newtheorem{proposition}[theorem]{Proposition}

\newtheorem*{theorem*}{Theorem}
\newtheorem*{proposition*}{Proposition}
\newtheorem*{conjecture*}{Conjecture}

\theoremstyle{definition}
\newtheorem{definition}[theorem]{Definition}
\newtheorem{example}[theorem]{Example}
\newtheorem*{remark}{Remark}
\newtheorem{notation}[theorem]{Notation}
\newtheorem{convention}[theorem]{Convention}

\makeatletter
\newcommand{\subjclass}[2][1991]{
  \let\@oldtitle\@title
  \gdef\@title{\@oldtitle\footnotetext{#1 \emph{Mathematics Subject Classification.} #2}}
}
\newcommand{\keywords}[1]{%
  \let\@@oldtitle\@title%
  \gdef\@title{\@@oldtitle\footnotetext{\emph{Key words and phrases.} #1.}}%
}
\makeatother

\title{Degree bounds for rational generators of invariant fields of finite abelian groups}
\author{Ben Blum-Smith}
\date{\today}


\begin{document}

\maketitle

\begin{abstract}
    We study degree bounds on rational but not necessarily polynomial generators for the field $\mathbf{k}(V)^G$ of rational invariants of a linear action of a finite abelian group. We show that lattice-theoretic methods used recently by the author and collaborators to study polynomial generators for the same field largely carry over, after minor modifications to the arguments. It then develops that the specific degree bounds found in that setting also carry over.
\end{abstract}

\section{Introduction}

Degree bounds for generators for rings of invariants are a topic of longstanding interest in invariant theory \cite{noether, fleischmann2000noether, fogarty2001noether, fleischmann2006noethermodular, symonds2011castelnuovo, schmid1991finite, domokos-hegedus, sezer2002sharpening, cziszter-domokos, derksen2017polynomial, gandini2019ideals, ferraro2021noether}. Much less studied are degree bounds for generators for fields of rational invariants. Very general upper bounds were proven in \cite{hubert-labahn} and \cite{fleischmann2007homomorphisms}. The author and collaborators took up this question systematically in \cite{blum2023degree}, considering the quantities
\[
\bfield(G,V) := \min(d:\KK(V)^G\text{ is generated by invariant polynomials of degree}\leq d)
\]
and
\[
\gfield(G,V) = \min(d: \KK(V)^G\text{ has a transcendence basis of polynomials of degree} \leq d),
\]
where $V$ is a representation of a group $G$ over a field $\KK$, and $\KK(V)^G$ is, as usual, the field of invariant rational functions. Results of that inquiry are discussed below.

The focus on {\em polynomials} generating the field of invariant rational functions was motivated by an application to signal processing: when $\KK=\RR$, the numbers $\bfield$ and $\gfield$ exert control over the number of samples needed to accurately estimate a signal in $V$ that is corrupted both by gaussian noise and by transformations selected randomly from $G$ (\cite[Theorems~2.15 and 2.16]{bandeira2017estimation}; and see  the introduction of \cite{blum2023degree}). Still, because $\KK(V)^G$ is a field, it is natural to ask what happens when we do not insist the field generators or transcendence basis are polynomials.

Therefore, in this note we consider versions of $\bfield$ and $\gfield$, called $\bfield^\mathbf{r}$ and $\gfield^\mathbf{r}$, which allow the generators of $\KK(V)^G$ (respectively, of a subfield of full transcendence degree) to be rational functions. This note is intended as a companion to \cite{blum2023degree}; while we make an effort to give self-contained statements of the results, the proofs inevitably refer to related arguments in \cite{blum2023degree}.

The results of \cite{blum2023degree} include sharp lower bounds depending on $V$ \cite[Theorems~3.1 and 3.2; Corollary~3.6]{blum2023degree}, and an upper bound in the case that $G=\ZZ/p\ZZ$ and $\operatorname{char}\KK\neq p$ \cite[Theorem~3.11]{blum2023degree}, provided $V$ has enough nontrivial isotypic components. This upper bound is not sharp; a sharp upper bound was conjectured \cite[Conjecture~5.1]{blum2023degree}. It is essentially immediate that the upper bound proven in \cite{blum2023degree} for $\bfield$ also applies to $\bfield^\mathbf{r}$; the main results of the present work are that, at least for abelian $G$ and non-modular $V$, the same lower bounds also apply. We also show that the upper bound conjectured in \cite{blum2023degree} for $\bfield$, is attained by  $\bfield^\mathbf{r}$ and even $\gfield^\mathbf{r}$. These results are given in Section~\ref{sec:bounds}. In Section~\ref{sec:field-gens-and-lattices}, we verify that the same lattice-theoretic methods used in \cite{blum2023degree} to study $\bfield,\gfield$ can be used to study $\bfield^\mathbf{r}$ and $\gfield^\mathbf{r}$. We also show that over any field bereft of nontrivial roots of unity (such as $\RR$), $\bfield$ and $\bfield^\mathbf{r}$ are  equal. The biggest difference with the theory developed in \cite{blum2023degree} is that $\bfield^\mathbf{r},\gfield^\mathbf{r}$ can be lowered by base-change. Notation and conventions are set up in Section~\ref{sec:notation}, where we also mention a useful (well-known) lemma.

\section{Notation and preliminaries}\label{sec:notation}

We import the following notation from \cite{blum2023degree}:
\begin{notation}
    Throughout, $G$ is a finite group, $\KK$ is a field, usually of characteristic coprime to $G$ (i.e., the {\em nonmodular case}), $V$ is a finite-dimensional representation of $G$ over $\KK$, and $N = \dim V$. The symbol $p$ always represents a prime number, specifically in the context that $G=\ZZ/p\ZZ$; likewise, $n$ is always a (not-necessarily prime) natural number such that $G=\ZZ/n\ZZ$. In the special case that $G$ is abelian and $\KK$ has coprime characteristic, $m$ represents the number of nontrivial isotypic components in $V$'s base change to $\KK$'s algebraic closure. (If $\KK$ already contains $|G|$th roots of unity, then $m$ is just the number of nontrivial isotypic components in $V$.)
\end{notation}

Before beginning the inquiry, one first has to decide on a definition of the degree of a rational function $f=f_1/f_2$ that generalizes the degree of polynomials. Reasonable choices include:
\begin{enumerate}
    \item $\max(\deg f_1,\deg f_2)$, so that, in the univariate case with $\KK=\CC$, the notion coincides with the topological degree of the self-map of the Riemann sphere induced by $f$.
    \item $\deg f_1 - \deg f_2$, so that, restricting to homogeneous rational functions, the degree captures the weight of the scaling action of $\KK^\times$ on $f$, i.e., 
    \[
    f(\alpha\bfx)=\alpha^{\deg f} f(\bfx)
    \]
    for $\alpha \in \KK^\times$.
    \item $\deg f_1 + \deg f_2$, so that $\deg f_1f_2 = \deg f_1f_2^{-1}$ and the notion provides a rough measure of the complexity of writing $f$ down.
\end{enumerate}

In this note we opt for the last of these choices. Thus, henceforward, we make the following definitions. 

\begin{definition}\label{def:degree-rational}
    For $f\in \KK(V)$,  the {\em degree} of $f$ is
\[
    \deg f := \deg f_1 + \deg f_2,
\]
where $f=f_1/f_2$ with $f_1,f_2\in \KK[V]$ and $\gcd(f_1,f_2) = 1$. Then $\KK(V)^G_{\leq d}$ is the set of invariant rational functions of degree $\leq d$. (Note that this is not a $\KK$-linear space, as it is not closed under addition.)
\end{definition}

\begin{definition}
With this in place, we set
\[
\bfield^\mathbf{r}(G,V):= \min(d:\KK(V)^G\text{ is generated as a field by }\KK(V)^G_{\leq d})
\]
and
\[
\gfield^\mathbf{r}(G,V):=\min(d:\KK(V)^G_{\leq d}\text{ contains a transcendence basis for }\KK(V)^G\text{ over }\KK),
\]
in analogy to $\bfield$ and $\gfield$ above.
\end{definition}

In the situation that $G$ is abelian, the characteristic of $\KK$ is coprime with $|G|$, and $\KK$ contains $|G|$th roots of unity, the action of $G$ on $V$ can be diagonalized. Working in a diagonal basis $x_1,\dots,x_N$, the action of $G$ on $V$ is monomial, i.e., every monomial is an eigenvector of the action, and the invariant monomials form a monoid under multiplication. Following a standard method in the invariant theory of finite abelian groups, we can then identify the invariant ring $\KK[V]^G$ with the monoid algebra, which is a normal affine semigroup ring. 

\begin{definition}\label{def:L(G,V)}
As in \cite{blum2023degree}, for abelian $G$ and $\KK$ containing distinct $|G|$th roots of unity, we write $L(G,V)$ for the ambient group of the monoid of invariant monomials, identified in the natural way with a sublattice of $\ZZ^N$. Explicitly, suppose that $G$ acts diagonally on a basis $x_1,\dots,x_N$ for $V^*$. Then if $\bfa\in \ZZ^N$, then $\bfx^\bfa$ is the Laurent monomial in $x_1,\dots,x_N$ with exponent vector $\bfa$, and 
\[
L(G,V) := \{\bfa \in \ZZ^N : g\bfx^\bfa = \bfx^\bfa\text{ for all }g\in G\}.
\]
\end{definition}

\begin{notation}
    With $\bfa\in\ZZ^N$ and $\bfx^\bfa\in \KK(V)$ as in Definition~\ref{def:L(G,V)}, we write
    \[
    \operatorname{exp}(\bfa):= \bfx^\bfa
    \]
    and view $\operatorname{exp}$ as a map from $\ZZ^N$ to $\KK(V)$. This notation is inspired by \cite[Chapter~6]{bruns-herzog}.
\end{notation}

\begin{definition}
It will transpire below, much as in \cite{blum2023degree}, that in the situation of Definition~\ref{def:L(G,V)}, $\gfield^\mathbf{r}(G,V)$ and $\bfield^\mathbf{r}(G,V)$ are not affected by deletion of duplicate or trivial characters in $V$. Following \cite{blum2023degree}, we write
\[
L(G,\Supp V) := L(G,V'),
\]
where $V'$ is a representation of $G$ obtained from $V$ by deleting all trivial and duplicate characters (so that each isotypic component of $V'$ is one-dimensional, and none is trivial). Note that $\dim V'=m$.
\end{definition}

\begin{convention}
In the situation of Definition~\ref{def:L(G,V)}, so that $L(G,V)$ and $L(G,\Supp V)$ are defined, we apply the word {\em degree} to points of these lattices so as to accord with Definition~\ref{def:degree-rational}. Explicitly, if $\bfa\in \ZZ^N$ or $\ZZ^m$, we set $\deg \bfa :=\deg \bfx^\bfa$. This is nothing but the $L^1$ norm of $\bfa$.  
\end{convention}

\begin{convention}\label{conv:chars-are-ints}
    In the situation of Definition~\ref{def:L(G,V)}, suppose, furthermore, that $G = \ZZ/n\ZZ$ with $n$ a natural number. Then we adopt the convention that an integer $k\in \ZZ$ represents the character
    \[
    \ell \mapsto e^{2\pi k\ell / n}\in \CC^\times
    \]
    of $G$, where $\ell\in G=\ZZ/n\ZZ$ (noting that the exponential on the right side is well-defined on residue classes mod $n$). With this convention, if $S$ is a set of integers representing the distinct nontrivial characters in $V^*$, it will follow from the work below that $\bfield^\mathbf{r}$ and $\gfield^\mathbf{r}$ only depend on the set $S$. Thus we write
    \[
    \gfield^\mathbf{r}(G,S) := \gfield^\mathbf{r}(G,V)\text{ and }\bfield^\mathbf{r}(G,S) := \bfield^\mathbf{r}(G,V).
    \]
    Furthermore, $L(G,\Supp V)$ is defined by the linear equation mod $n$ with coefficients from $S$. Therefore, given an arbitrary set of integers $S=\{s_1,\dots,s_m\}$, we write
    \[
    L(G,S) := \{\bfa\in \ZZ^m : s_1a_1 + \dots + s_ma_m = 0\pmod{n}\}.
    \]
\end{convention}

These conventions having been set up, we mention the following well-known fact that will be used repeatedly:

\begin{lemma}\label{lem:semi-invariants}
    If $f=f_1/f_2\in \KK(V)^G$ is an invariant rational function, with $f_1,f_2\in \KK[V]$ coprime, then $f_1,f_2$ are semi-invariants of $G$ of the same weight; in other words, there exists a  character
    \[
    \varepsilon: G\rightarrow \KK^\times
    \]
    such that $gf_i=\varepsilon(g)f_i$ for $i=1,2$ and all $g\in G$. 
\end{lemma}

\begin{proof}
    This is standard (see, e.g.,  \cite[Theorem~3.3(b)]{popov-vinberg} and \cite[Lemma~3.2]{hubert-labahn}), but the elementary proof is short enough to include:

    For any $g\in G$, we have $f_1/f_2=gf_1/gf_2$ by the invariance of $f$, or in other words,
    \[
f_1(gf_2) = f_2(gf_1).
    \]
    Since $f_1$ and $f_2$ are coprime and $\KK[V]$ is a unique factorization domain, it must be that $f_1$ divides $gf_1$. Then  $f_2$ divides $gf_2$ with the same ratio. Since $g$ is degree-preserving, the common ratio is in $\KK^\times$. We define
    \[
    \varepsilon(g) := \frac{gf_1}{f_1} = \frac{gf_2}{f_2}.
    \]
    It is immediate that $\varepsilon: G\rightarrow \KK^\times$ is a $1$-cocycle:
    \[
    \varepsilon(gh) = \frac{ghf_1}{f_1} = \frac{gf_1}{f_1}\cdot \frac{ghf_1}{gf_1} = \frac{gf_1}{f_1}\cdot g\left(\frac{hf_1}{f_1}\right) = \varepsilon(g)\cdot g\varepsilon(h).
    \]
    Since the $G$-action on $\KK^\times$ is trivial, it is actually a group homomorphism.
\end{proof}

\section{Rational field generators and lattices}\label{sec:field-gens-and-lattices}

We now give a preliminary exploration of $\bfield^\mathbf{r}$ and $\gfield^\mathbf{r}$ by adapting the methods of \cite{blum2023degree}. A first observation is that all the upper bounds on $\bfield,\gfield$ proven in \cite{blum2023degree} hold in the present setting: the basic inequalities
\begin{align*}
\bfield^\mathbf{r}(G,V) &\leq \bfield(G,V),\\ \gfield^\mathbf{r}(G,V) &\leq \gfield(G,V)
\end{align*}
are immediate from the definitions, because $\KK[V]^G_{\leq d} \subset \KK(V)^G_{\leq d}$. It is also evident that
\[
\gfield^\mathbf{r}(G,V) \leq \bfield^\mathbf{r}(G,V),
\]
so lower bounds on $\gfield^\mathbf{r}$ apply to $\bfield^\mathbf{r}$, and upper bounds on $\bfield^\mathbf{r}$ apply to $\gfield^\mathbf{r}$. In particular, by \cite[Theorem~3.11]{blum2023degree}, if $G=\ZZ/p\ZZ$,  $\KK$ has characteristic coprime with $p$, and $V$ contains at least three nontrivial characters of $G$ after base-changing to the algebraic closure of $\KK$, then we have
\[
\gfield^\mathbf{r}(G,V) \leq \bfield^\mathbf{r}(G,V), \gfield(G,V) \leq \bfield(G,V) \leq \frac{p+3}{2}.
\]
If \cite[Conjecture~5.1]{blum2023degree} is correct, the right side can be lowered to $\lceil p / \lceil m/2\rceil\rceil$. This latter is sharp for $\gfield$ \cite[Proposition~5.2]{blum2023degree}, and we will see below (Proposition~\ref{prop:conjecture-is-sharp}) that it is even sharp for $\gfield^\mathbf{r}$ .

Next, in the situation that $\KK=\RR$, admitting rational field generators does not allow one to reduce the degrees of the functions needed to generate the field:

\begin{proposition}\label{prop:rational-no-better-over-R}
    Let $G$ be a finite group and $V$ a finite-dimensional representation of $G$ over a field $\KK$ whose only roots of unity are $\pm 1$. Then
    \[
    \bfield^\mathbf{r}(G,V) = \bfield(G,V).
    \]
\end{proposition}

The main idea in the argument is that the paucity of roots of unity forces any rational invariant to be representable in terms of polynomial invariants of equally low degree.

\begin{proof}
    It was noted above that $\bfield^\mathbf{r}(G,V) \leq\bfield(G,V)$. We need the reverse inequality.

    Let $f=f_1/f_2$ be any invariant rational function, with $f_1,f_2\in \KK[V]$ and $\gcd(f_1,f_2)=1$. By Lemma~\ref{lem:semi-invariants}, there exists a group homomorphism $\varepsilon:G\rightarrow \KK^\times$ with $gf_i = \varepsilon(g)f_i$ for $i=1,2$ and all $g\in G$.
    Since $G$ is finite (so that $\varepsilon(g)$ must be a root of unity for all $g\in G$), and we have assumed that the only roots of unity in $\KK$ are $\pm 1$, we have $\epsilon(g)^2=1$ for all $g$. It follows that in addition to $f = f_1/f_2$, all of 
    \[
    f_1^2,\; f_1f_2,\; f_2^2
    \]
    are invariant. Furthermore, according to our definition of $\deg f$, it is equal to $\deg f_1f_2$ and also to the mean of $\deg f_1^2,\deg f_2^2$; thus at least one of $f_1^2,f_2^2$ has  degree at most that of $f$.

    It follows that for a rational invariant $f$ of degree $d$, either $f_1^2/(f_1f_2)$ or $f_1f_2/f_2^2$ yields a representation of $f$ as a ratio of invariant polynomials of degree $\leq d$. So for any given degree $d$, and in particular, for $d=\bfield^\mathbf{r}(G,V)$, the invariant rational functions of degree $\leq d$ are already in the field generated by the invariant polynomials of degree $\leq d$. We can conclude that 
    \[
    \bfield(G,V) \leq  \bfield^\mathbf{r}(G,V),
    \]
    completing the proof.
\end{proof}

If $\KK$ is a field richer in roots of unity than $\RR$, the conclusion of Proposition~\ref{prop:rational-no-better-over-R} may fail; we can have $\bfield^\mathbf{r}$ properly lower than $\bfield$.

\begin{example}\label{ex:base-change-example}
    Let $\KK=\CC$ and let $G = \ZZ/7\ZZ$. Let $V$ be the $3$-dimensional representation given by the characters $1$, $2$, and $4$ (as in Convention~\ref{conv:chars-are-ints}). Then $L(G,V)$ contains the point $(1,1,1)$ but no other nontrivial, non-negative points of degree $\leq 3$. On the other hand, it contains $(1,3,0)$ and $(0,1,3)$ which, together with $(1,1,1)$, generate it. Thus, $L(G,V)$ is generated by points of degree $4$ but the points of lower degree do not generate a full-rank sublattice. By the results of \cite[Section~2]{blum2023degree}, it follows that $\gfield(G,V) = \bfield(G,V) = 4$. 
    
    Meanwhile, if we take $x_1,x_2,x_4$ as the diagonal basis for $V^*$ corresponding to the characters $1,2,4$ respectively, then $x_1^2/x_2$ and $x_2^2/x_4$ are rational invariants of degree $3$, corresponding to the lattice points $(2,-1,0)$ and $(0,2,-1)$. Together with $x_1x_2x_4$, these generate the field of rational invariants. Thus $\bfield^\mathbf{r}(G,V)\leq 3$. (Using the tools below, this can be sharpened to equality.)
\end{example}

To continue the inquiry, we need analogues to the results of \cite[Section~2]{blum2023degree}, relating $\bfield$ and $\gfield$ to $L(G,V)$, and using this to establish that they only depend on the set of nontrivial characters in $V^*$ when $G$ is abelian and $\KK$ has coprime characteristic and sufficiently many roots of unity. The analogue of \cite[Lemma~2.1]{blum2023degree} (allowing arbitrary base changes without affecting $\bfield$ or $\gfield$) actually fails, but if we assume that the ground field contains all the relevant roots of unity, the rest of \cite[Section~2]{blum2023degree} goes through. Furthermore, base change can still only lower $\bfield^\mathbf{r}$ and $\gfield^\mathbf{r}$, so we will be able to get a version of the lower bounds of \cite{blum2023degree} in this setting, without this restriction on the ground field.

We illustrate the failure of the base-change lemma for $\bfield^\mathbf{r}(G,V)$ with an example.

\begin{example}\label{ex:base-change-fails}
    Let $\KK:=\QQ(\sqrt{-7})$. This is the subfield of the cyclotomic field $\LL:=\QQ(\zeta_7)$ corresponding to the subgroup $H$ of order three in $\operatorname{Gal}(\LL/\QQ) \cong (\ZZ/7\ZZ)^\times$. The Galois group $\operatorname{Gal}(\LL/\QQ)$ acts on the character group of $G=\ZZ/7\ZZ$ by post-composition, and $\{1,2,4\}$ is an orbit for the action of $H$. Thus the representation of $G=\ZZ/7\ZZ$ in Example~\ref{ex:base-change-example} is in fact defined over $\KK=\LL^H$; let $V$ be the corresponding representation over $\KK$, and let $V_\LL$ be its base-change to $\LL$. Because $\LL$ contains $7$th roots of unity, the diagonal basis $x_1,x_2,x_4$ discussed in Example~\ref{ex:base-change-example} is defined over $\LL$; thus the argument of Example~\ref{ex:base-change-example} shows that $\bfield^\mathbf{r}(G,V_\LL) \leq 3$ (and again, this can be sharpened to equality using the tools below). 
    
    On the other hand, we claim that $\bfield^\mathbf{r}(G,V) = 4$.  Because $\KK$'s only roots of unity are $\pm 1$, we have $\bfield^\mathbf{r}(G,V)=\bfield(G,V)$ by  Proposition~\ref{prop:rational-no-better-over-R}; then, we have $\bfield(G,V)=4$ by Example~\ref{ex:base-change-example} combined with \cite[Lemma~2.1]{blum2023degree}. This establishes the claim. So $\bfield^\mathbf{r}(G,V_\LL)<\bfield^\mathbf{r}(G,V)$ strictly here.
\end{example}

Although the base-change lemma fails, its trivial half continues to hold. The argument is essentially equivalent to the corresponding half of \cite[Lemma~2.1]{blum2023degree}, and comes down to the fact that $\KK(V)^G_{\leq d} \subset \LL(V_\LL)^G_{\leq d}$.

\begin{lemma}\label{lem:half-base-change}
    If $G$ is a finite group, $V$ is a finite-dimensional representation of $G$ over a field $\KK$, and $\LL$ is an arbitrary extension of $\KK$, and $V_\LL := \LL\otimes_\KK V$, then
    \[
    \bfield^\mathbf{r}(G,V_\LL) \leq \bfield^\mathbf{r}(G,V)
    \]
    and
    \[
    \gfield^\mathbf{r}(G,V_\LL) \leq \gfield^\mathbf{r}(G,V).
    \]
\end{lemma}

\begin{proof}
    Let $d= \bfield^\mathbf{r}(G,V)$. Then $\KK(V)^G_{\leq d}$ generates $\KK(V)^G$ as a field. Since $\KK(V)^G_{\leq d} \subset \LL(V_\LL)^G_{\leq d}$, it follows that the field generated by $\LL(V_\LL)^G_{\leq d}$ contains both $\LL$ and $\KK[V]^G\subset \KK(V)^G$. It thus contains $\LL[V_\LL]^G \cong \LL\otimes_\KK \KK[V]^G$, and therefore its fraction field $\LL(V_\LL)^G$. Thus, $\bfield^\mathbf{r}(G,V_\LL)\leq d$.

    Similarly, if $d=\gfield^\mathbf{r}(G,V)$, then $\KK(V)^G_{\leq d}$ contains a transcendence basis for $\KK(V)$ over $\KK$. The elements of this basis are also contained in $\LL(V_\LL)^G_{\leq d} \supset \KK(V)^G_{\leq d}$. Just as in the proof of \cite[Lemma~2.1]{blum2023degree}, they remain algebraically independent over $\LL$. Thus, by counting, they constitute a transcendence basis of $\LL(V_\LL)$ over $\LL$, so $\gfield^\mathbf{r}(G,V_\LL)\leq d$.
\end{proof}

Going forward, we assume that $G$ is abelian, that $\operatorname{char}\KK \nmid |G|$, and that $\KK$ contains $|G|$th roots of unity, so that the action of $G$ on $V$ is diagonalizable over $\KK$. We  work in a diagonal basis $x_1,\dots,x_N$ for $V^*$; the lattice $L(G,V)$ is then defined. By Lemma~\ref{lem:half-base-change}, lower bounds on $\bfield^\mathbf{r},\gfield^\mathbf{r}$ obtained in this setting will apply over arbitrary non-modular $\KK$. As mentioned above, under these hypotheses, all the remaining results in \cite[Section~2]{blum2023degree} go through, as we now verify.
\begin{notation}
    For positive integer $N$, let $\Omega_N$ be the standard cross-polytope in $\RR^N$, the convex hull of the standard basis vectors $\bfe_1,\dots,\bfe_N$ and their negatives. It is alternatively characterized as the closed unit ball in the $L^1$ norm on $\RR^N$.\footnote{A more standard name for this polytope is $\beta_N$; however, we opt for $\Omega_N$ to avoid overworking the symbol $\beta$ any more than necessary.}
\end{notation}
This polytope takes the place of the $\Delta_N$ of \cite[Section~2]{blum2023degree}, which is the convex hull of $\bfe_1,\dots,\bfe_N$ and the origin. For any $d>0$, the points of $\ZZ^N$ contained in $d\Omega_N$ are the exponent vectors of the Laurent monomials of degree $\leq d$.

We have an analogue of \cite[Lemma~2.4]{blum2023degree}:

\begin{lemma}\label{lem:pre-equivalence-rational}
    Let $d$ be a positive integer.
    \begin{enumerate}
        \item The points of $L(G,V)$ contained in $d\Omega_N$ generate $L(G,V)$ if and only if $\KK(V)^G_{\leq d}$ generates $\KK(V)^G$ as a field.
        \item The points of $L(G,V)$ contained in $d\Omega_N$ generate a full-rank sublattice of $L(G,V)$ if and only if $\KK(V)^G$ is a finite field extension of the field generated by $\KK(V)^G_{\leq d}$. Furthermore, when these equivalent conditions hold, the degree of the field extension equals the index of the lattice containment, i.e., 
        \[
        [\KK(V)^G:\KK(\KK(V)^G_{\leq d})] = [L(G,V): \langle L(G,V)\cap d\Omega_N\rangle].
        \]
    \end{enumerate}
\end{lemma}

\begin{proof}
    We need one significant idea not already found in the proof of \cite[Lemma~2.4]{blum2023degree}:
    
    \begin{lemma}\label{lem:rationally-expressible}
        Every element of $\KK(V)^G_{\leq d}$ is rationally expressible in terms of the Laurent monomials contained in $\KK(V)^G_{\leq d}$.
    \end{lemma} 
    
    \begin{proof}
    Let $f = f_1/f_2\in \KK(V)^G_{\leq d}$ be a rational invariant of degree $\leq d$, with $f_1,f_2$ coprime. By Lemma~\ref{lem:semi-invariants}, $f_1,f_2$ are semi-invariants of a common weight $\varepsilon:G\rightarrow \KK^\times$. Let $f_1 = \sum_i m_i$ be the decomposition of $f_i$ into distinct monomials $m_i$ (in the diagonal basis $x_1,\dots,x_N$) and let $f_2 = \sum_j n_j$ be the corresponding decomposition of $f_2$. We have 
    \[
    \deg m_i \leq \deg f_1,\; \deg n_j \leq \deg f_2
    \]
    for any $i,j$. Because $G$ acts diagonally, each $m_i$ is a semi-invariant for the action of $G$; linear independence of the $m_i$ over $\KK$ then implies that all of them have weight $\varepsilon$; by the same argument, all the $n_j$ are semi-invariants of weight $\varepsilon$. It follows that for any $i,j$, the Laurent monomial $m_in_j^{-1}$ is invariant. Furthermore, we have
    \[
    \deg m_in_j^{-1} = \deg m_i + \deg n_j \leq \deg f_1 + \deg f_2 \leq d.
    \]
    Putting this together, we have $m_in_j^{-1}\in \KK(V)^G_{\leq d}$ for any $i,j$. Now  we reach the desired conclusion by applying the following elementary calculation, which represents a ratio of sums as a rational function of the ratios of the individual terms:
    \[
    f = \frac{f_1}{f_2} = \frac{\sum_i m_i}{\sum_j n_j} = \sum_i\left(\sum_j\left(\frac{m_i}{n_j}\right)^{-1}\right)^{-1}.\qedhere
    \]
    \end{proof}
    
    With this in place, the proof of \cite[Lemma~2.4]{blum2023degree} goes through {\em mutatis mutandis}, as follows. After replacing $\Delta_N$ and $\KK[V]^G_{\leq d}$ everywhere with $\Omega_N$ and $\KK(V)^G_{\leq d}$ respectively, the proof transfers to the present context word-for-word except for two points which require to invoke the above:
    \begin{enumerate}
        \item In contrast to $\KK[V]^G_{\leq d}$, it is not the case that $\KK(V)^G_{\leq d}$ is the linear span of the Laurent monomials it contains. But by Lemma~\ref{lem:rationally-expressible}, it is contained in the field generated by the Laurent monomials it contains. In particular, it contains a transcendence basis for $\KK(V)^G$ over $\KK$ if and only if the set of Laurent monomials inside it contains a transcendence basis. So the same argument as in \cite[Lemma~2.4]{blum2023degree} shows that $\KK(V)^G$ is finite over the subfield generated by $\KK(V)^G_{\leq d}$ if and only if the lattice generated by $d\Omega_N \cap L(G,V)$ is full-rank.

        \item The proof of \cite[Lemma~2.4]{blum2023degree} features the following sequence of containments, where $L'$ is the group $\exp(\langle L(G,V)\cap d\Delta_N\rangle)$ of Laurent monomials generated by the invariant (true) monomials of degree $\leq d$:
        \[
        \KK[\KK[V]^G_{\leq d}] \subset \KK[L'] \subset \KK(\KK[V]^G_{\leq d}).
        \]
        The first containment becomes false when we replace $\KK[V]^G_{\leq d}$ with $\KK(V)^G_{\leq d}$ and redefine $L'$ as \[
        \exp(\langle L(G,V)\cap d\Omega_N\rangle),
        \]
        the group of Laurent monomials generated by the invariant {\em Laurent} monomials of degree $\leq d$. However, the only purpose for this containment in the proof of \cite[Lemma~2.4]{blum2023degree} was to argue that $\KK(L') = \KK(\KK[V]^G_{\leq d})$, and it is true that
        \[
        \KK(L') = \KK(\KK(V)^G_{\leq d})
        \]
        with the new definition of $L'$. This follows from Lemma~\ref{lem:rationally-expressible} because $\exp(L(G,V) \cap d\Omega_N)$ is exactly the set of Laurent monomials contained in $\KK(V)^G_{\leq d}$, so Lemma~\ref{lem:rationally-expressible} implies equality of the first and last terms in the following evident sequence of containments:
        \[
        \KK(\exp(L(G,V) \cap d\Omega_N)) \subset \KK(\exp(\langle L(G,V)\cap d\Omega_N\rangle))=\KK(L') \subset \KK(\KK(V)^G_{\leq d}).
        \]
        Having established $\KK(L') = \KK(\KK(V)^G_{\leq d}))$ in this setting, the rest of the proof is unaffected.\qedhere
    \end{enumerate}
\end{proof}

The analogue of \cite[Lemma~2.6]{blum2023degree} is now immediate from Lemma~\ref{lem:pre-equivalence-rational} and the definitions of $\gfield^\mathbf{r}$ and $\bfield^\mathbf{r}$:

\begin{lemma}\label{lem:equivalence-rational}
    We have
    \[
    \bfield^\mathbf{r}(G,V) = \min(d:L(G,V) = \langle L(G,V)\cap d\Omega_N\rangle)
    \]
    and
    \[
    \pushQED{\qed} 
    \gfield^\mathbf{r}(G,V) = \min(d:\operatorname{rk}\langle L(G,V)\cap d\Omega_N\rangle = N). \qedhere
    \popQED
    \]
\end{lemma}

We also have analogues of \cite[Lemma~2.9]{blum2023degree}, \cite[Lemma~2.11]{blum2023degree}, and therefore \cite[Lemma~2.12]{blum2023degree}.

\begin{lemma}\label{lem:delete-triv-rational}
    Let $V'$ be a representation of $G$ obtained from $V$ by deleting a trivial character. Then $\bfield^\mathbf{r}(G,V) = \bfield^\mathbf{r}(G,V')$ and $\gfield^\mathbf{r}(G,V) = \gfield^\mathbf{r}(G,V')$.
\end{lemma}

\begin{lemma}\label{lem:merge-identical-rational}
    Let $V'$ be a representation of $G$ obtained from $V$ by merging a pair of identical characters. Then $\bfield^\mathbf{r}(G,V)=\bfield^\mathbf{r}(G,V')$ and $\gfield^\mathbf{r}(G,V)=\gfield^\mathbf{r}(G,V')$.
\end{lemma}

\begin{proof}[Proof of Lemmas~\ref{lem:delete-triv-rational} and \ref{lem:merge-identical-rational}]
    The proofs of \cite[Lemma~2.9]{blum2023degree} and \cite[Lemma~2.11]{blum2023degree} go through essentially word-for-word in this setting after changing $\Delta_N$ to $\Omega_N$ and replacing the calls to \cite[Lemma~2.6]{blum2023degree} with calls to Lemma~\ref{lem:equivalence-rational}. The key observation is that the embedding $I$ of those proofs injects $d\Omega_{N-1}$ into $d\Omega_N$ (just as it injects $d\Delta_{N-1}$ into $d\Delta_N$), and likewise, the maps $\pi$ and $\pi'$ (of \cite[Lemma~2.9]{blum2023degree} and \cite[Lemma~2.11]{blum2023degree} respectively) map $d\Omega_N$ onto $d\Omega_{N-1}$ (just as they map $d\Delta_N$ onto $d\Delta_{N-1}$).
\end{proof}

\begin{lemma}\label{lem:distinct-nontrivial-rational}
    The numbers $\bfield^\mathbf{r}(G,V)$ and $\gfield^\mathbf{r}(G,V)$ depend only on the set of distinct, nontrivial characters of $G$ that occur in $V$ (and not on their multiplicities).
\end{lemma}

\begin{proof}
    Immediate by induction from Lemmas~\ref{lem:delete-triv-rational} and \ref{lem:merge-identical-rational}.
\end{proof}

\section{Bounds}\label{sec:bounds}

This groundwork having been laid, we can reproduce all of the main results of \cite{blum2023degree} for abelian groups in the present setting: 

\begin{theorem}[Sharp lower bound]\label{thm:lower-bound-rational}
If $G$ is a finite abelian group, and $V$ is a faithful, non-modular, finite-dimensional representation of $G$, and $m$ is the number of distinct, nontrivial characters occurring in $V$, then
\[
\gfield^\mathrm{r}(G,V)\geq \sqrt[m]{|G|}.
\]
\end{theorem}

\begin{proof}
First assume that $\KK$ contains the $|G|$th roots of unity, so we can work in the diagonal basis. The proof of \cite[Theorem~3.1]{blum2023degree} goes through essentially unchanged in this setting (after replacing $\gfield$ with $\gfield^\mathbf{r}$), since none of the inequalities in that proof actually require the hypothesis that the points $\bfa_1,\dots,\bfa_m$ lie in the first orthant, now that we have expanded the definition of $\deg \bfa_i$ to cover all orthants.

If $\KK$ does not contain all the distinct $|G|$th roots of unity, let $\LL$ be a field extension of $\KK$ that does (which exists since by hypothesis, $\operatorname{char}\KK$ does not divide $|G|$). Then $\gfield^\mathbf{r}(G,V)\geq \gfield^\mathbf{r}(G,V_\LL)$ by Lemma~\ref{lem:half-base-change}, and $\gfield^\mathbf{r}(G,V_\LL)\geq \sqrt[m]{|G|}$ by the previous paragraph, so we have the desired conclusion.
\end{proof}

\begin{proposition}[Characterization of groups attaining the lower bound]
Suppose equality is attained in Theorem~\ref{thm:lower-bound-rational}, and let $d:=\gfield^\mathbf{r}(G,V)$. Then $\KK$ contains $d$th roots of unity, $G\cong (\ZZ/d\ZZ)^m$, and the distinct, nontrivial characters in $V$ form a basis for the character group $\widehat G$ of $G$ (as a $\ZZ/d\ZZ$-module).
\end{proposition}

\begin{proof}
The proof of \cite[Proposition~3.4]{blum2023degree} goes through {\em mutatis mutandis}. This time we get that the $\bfa_i$ lie on coordinate axes (assertion~2 in that proof) from the fact that $|\bfa_i|=\deg \bfa_i$ (which must hold for each $i$ in order for equality to be attained in all the inequalities). In assertion~3, we may get $\bfa_i = -d\bfe_i$ for some of the $i$; if so, replace $\bfa_i$ with $-\bfa_i$. The rest of the argument goes through word-for-word in this context.
\end{proof}

\begin{proposition}[``Hard floor" lower bound]\label{prop:gfield-at-least-3-rational}
    Let $G$ be a finite abelian group, and $V$ a nontrivial, non-modular, finite-dimensional representation of $G$. Then $\gfield^\mathbf{r}=2$ if and only if all the nontrivial characters in $V$ are involutions; otherwise, it is at least $3$.

    In particular, if $m$ is the number of distinct nontrivial characters in $V$ and $\tau$ is the number of involutions in $G$, then the condition
    \[
    m>\tau
    \]
    implies that
    \[
    \gfield^\mathbf{r}(G,V)\geq 3.
    \]
\end{proposition}

\begin{proof}
    As in the proof of Theorem~\ref{thm:lower-bound-rational}, we can assume that $\KK$ contains distinct $|G|$th roots of unity since base-changing to an algebraic closure can only make $\gfield^\mathbf{r}$ smaller, per Lemma~\ref{lem:half-base-change}, and we are proving a lower bound.

    The argument in the proof of \cite[Proposition~3.5]{blum2023degree} goes through {\em mutatis mutandis} after observing that even with the expanded meaning of degree used in this note which applies to points in all orthants, $L(G,\Supp V)$ still contains no points of degree $1$ or $2$ other than those contemplated in the proof of \cite[Proposition~3.5]{blum2023degree} and their opposites. This is evident for degree $1$, and for degree $2$ it is because if a point $\bfa=(\bfa_\chi)_{\chi \in \Supp V}$ is of degree $2$ and has two coordinates of different signs, then it must be of the form $a_{\chi_1}=1$, $a_{\chi_2}=-1$, and $a_\chi= 0$ for $\chi\neq \chi_1,\chi_2$, for characters $\chi_1,\chi_2\in \Supp V$. However, all the characters in $\Supp V$ are distinct (by definition), and in particular, $\chi_1\neq \chi_2$. Thus such a point cannot satisfy the equation $\chi^\bfa = 1\in \widehat G$ that defines $L(G,\Supp V)$. Thus any $\bfa$ of degree $1$ or $2$ in $L(G,\Supp V)$ has all coordinates non-negative or all coordinates non-positive, i.e., they are exactly the points contemplated in the proof of \cite[Proposition~3.5]{blum2023degree} together with their opposites. Since adding in the opposites of a set of first-orthant points has no effect on the lattice they generate, the counting argument in \cite[Proposition~3.5]{blum2023degree} now goes through.
\end{proof}

\begin{corollary}
    If $G$ is finite abelian but not an elementary abelian $2$-group, and $V$ is a faithful non-modular finite-dimensional representation, then $\gfield(G,V)\geq 3$.
\end{corollary}

\begin{proof}
    Again we can assume $\KK$ contains the $|G|$th roots of unity, as in Theorem~\ref{thm:lower-bound-rational} and Proposition~\ref{prop:gfield-at-least-3-rational}. Then the proof of \cite[Corollary~3.6]{blum2023degree} goes through word-for-word after changing $\gfield$ to $\gfield^\mathbf{r}$, and replacing the call to \cite[Proposition~3.5]{blum2023degree} with a call to Proposition~\ref{prop:gfield-at-least-3-rational}.
\end{proof}

As mentioned above, we also get versions of the upper bounds, with $\bfield^\mathbf{r}$ in the place of $\bfield$, because of the basic inequality $\bfield^\mathbf{r}(G,V)\leq \bfield(G,V)$. Thus, by \cite[Theorem~3.11]{blum2023degree}, for $G=\ZZ/p\ZZ$,  $\KK$ of coprime characteristic, and $V$ such that its base change to the algebraic closure of $\KK$ contains at least 3 distinct nontrivial characters, we have
\[
\bfield^\mathbf{r}(G,V)\leq \frac{p+3}{2},
\]
and by \cite[Proposition~4.4]{blum2023degree}, the same bound holds if there are only two distinct nontrivial characters in $V$'s base change, as long as they are not inverse to each other.

Finally, a version of \cite[Proposition~5.2]{blum2023degree} goes through: 

\begin{proposition}\label{prop:conjecture-is-sharp}
Let $G=\ZZ/n\ZZ$ with $n\geq 3$, and choose any $1 \leq m < n$. If $m$ is even, define
\[
S_m := \{ \pm 1, \pm 2, \dots, \pm m/2\}\subset\widehat G,
\]
where characters of $G$ are represented by integers as in Convention~\ref{conv:chars-are-ints}. If $m$ is odd, define 
\[
S_m := \{ \pm 1, \pm 2, \dots, \pm (m-1)/2, (m+1)/2 \} \subset \widehat G.
\] 
In either case, we have
\[
\bfield^\mathbf{r}(G,S_m) = \gfield^\mathbf{r}(G,S_m) = \max\left(3,\left\lceil \frac{n}{\lceil m/2\rceil} \right\rceil\right).
\]
\end{proposition}

\begin{proof}
    The proof of \cite[Proposition~5.2]{blum2023degree} works word-for-word after replacing $a_A$ with $|a_A|$ in the inequality
\[
|\varphi(a)| \leq \sum_{A\in S_m} |A|a_A \leq \lceil m/2\rceil \deg \bfa.\qedhere
\]
\end{proof}

Therefore, $\bfield^\mathbf{r}$, and even $\gfield^\mathbf{r}$, attains the upper bound conjectured for $\bfield$ in \cite[Conjecture~5.1]{blum2023degree}.

\begin{remark}
    The proofs of \cite[Theorem~3.1]{blum2023degree} and Theorem~\ref{thm:lower-bound-rational} above are of the style of argumentation of Minkowski's classical {\em geometry of numbers} \cite{lekkerkerker-gruber}, which studies interactions between convex bodies and lattices by reasoning about lengths and volumes. One may ask what classical geometry of numbers can tell us about upper bounds. This is a particularly tempting question regarding $\gfield^\mathbf{r}$ and $\bfield^\mathbf{r}$, since $\Omega_m$ (unlike $\Delta_m$) is centrally symmetric in addition to being convex and bounded, so the original theorems of Minkowski can be applied directly. Furthermore,   Lemmas~\ref{lem:equivalence-rational} and \ref{lem:distinct-nontrivial-rational} imply that $\gfield^\mathbf{r}(G,V)$ is precisely the highest successive minimum (an object of classical study, see \cite[p.~123]{lekkerkerker-gruber}) of $\Omega_m$ with respect to $L(G,\Supp V)$. 
    
    One does obtain an upper bound on $\gfield^\mathbf{r}$ by na\"ively applying Minkowski's Second Theorem \cite[Chapter~2, \S~16, Theorem 3]{lekkerkerker-gruber}, also known as Minkowski's theorem on successive minima. However, it is much weaker than the upper bounds proven in \cite{blum2023degree}, or even the Noether bound \cite{noether, fleischmann2000noether, fogarty2001noether}, which states that the ring of invariants is generated by polynomials of degree $\leq |G|$ in coprime characteristic.
    
    To illustrate, take $G= \ZZ/p\ZZ$ with $p$ an odd prime, and $V$ a representation of $G$ over $\CC$ containing $m$ distinct, nontrivial characters. Minkowski's Second Theorem, applied to the lattice $L:=L(G,\Supp V)$ and the bounded, centrally symmetric convex body $\Omega_m$, states that 
    \[
    \lambda_1(\Omega_m,L)\cdots\lambda_m(\Omega_m,L) V(\Omega_m) \leq 2^m\det(L),
    \]
    where $\lambda_1(\Omega_m,L),\dots,\lambda_m(\Omega_m,L)$ are the successive minima of $\Omega_m$ with respect to $L$, the last of which is $\gfield^\mathbf{r}(G,V)$; $\det(L)$ is the determinant of the lattice $L$, which is $p$; and $V(\Omega_m)$ is the volume of $\Omega_m$, which is $2^m/m!$. We thus obtain
    \[
    \lambda_1(\Omega_m,L)\cdots\lambda_{m-1}(\Omega_m,L)\gfield^\mathbf{r}(G,V) \leq m!p.
    \]
    The arguments in the proofs of \cite[Proposition~3.5]{blum2023degree} and Proposition~\ref{prop:gfield-at-least-3-rational} imply that $L$ contains no points of $L^1$-norm 1, and (because $p$ is odd, so $G$ contains no involutions) at most $\lfloor m/2 \rfloor$ linearly independent points of $L^1$-norm $2$. Because $\Omega_m$ is the unit ball in the $L^1$ norm, it follows that $\lambda_1(\Omega_m,L),\dots,\lambda_{\lfloor m/2\rfloor}(\Omega_m,L)$ are all $\geq 2$, while $\lambda_{\lfloor m/2\rfloor+1}(\Omega_m,L),\dots,\lambda_{m-1}(\Omega_m,L)$ are all $\geq 3$. (These bounds are attained by the lattices constructed in \cite[Proposition~5.2]{blum2023degree}, so they are best possible.) Incorporating these bounds, we finally obtain
    \begin{equation}\label{eq:successive-minima-bound}
\gfield^\mathbf{r}(G,V) \leq \frac{m!p}{2^{\lfloor m/2\rfloor}3^{\lceil m/2\rceil - 1}}.
    \end{equation}
    For $m=1,2$, the right side is $p$, which is sharp; but for $m\geq 3$, the bound
    \[
    \gfield^\mathbf{r}(G,V) \leq \bfield^\mathbf{r}(G,V) \leq \bfield(G,V) \leq \frac{p+3}{2}
    \]
    obtained from \cite[Theorem~3.11]{blum2023degree} is better, and for large $m$ it is dramatically better. Indeed, the right side of \eqref{eq:successive-minima-bound} increases rapidly with $m$, while \cite[Theorem~3.11]{blum2023degree} and the Noether bound do not depend on $m$, and the true upper bound suggested by computational data \cite[Conjecture~5.1]{blum2023degree} decreases with increasing $m$.

    This does not rule out the possibility of proving a more useful upper bound via a more thoughtful application of the classical theorems.
\end{remark}

\section*{Acknowledgements}

The author wishes to thank Larry Guth for the encouragement to pursue the question addressed here, and Thays Garcia, Rawin Hidalgo, and Consuelo Rodriguez for the collaboration without which this work would not have been possible to conceive.

\bibliographystyle{alpha}
\bibliography{bib}

\newcommand{\etalchar}[1]{$^{#1}$}
\begin{thebibliography}{BSGHR24}

\bibitem[BBSK{\etalchar{+}}23]{bandeira2017estimation}
Afonso~S Bandeira, Ben Blum-Smith, Joe Kileel, Amelia Perry, Jonathan
  Niles-Weed, and Alexander~S Wein.
\newblock Estimation under group actions: recovering orbits from invariants.
\newblock {\em Applied and Computational Harmonic Analysis}, 66:236--319, 2023.

\bibitem[BH98]{bruns-herzog}
Winfried Bruns and H~J{\"u}rgen Herzog.
\newblock {\em Cohen-{M}acaulay rings}.
\newblock Number~39 in Cambridge studies in advanced mathematics. Cambridge
  university press, 1998.

\bibitem[BSGHR24]{blum2023degree}
Ben Blum-Smith, Thays Garcia, Rawin Hidalgo, and Consuelo Rodriguez.
\newblock Degree bounds for fields of rational invariants of
  $\mathbb{Z}/p\mathbb{Z}$ and other finite groups.
\newblock {\em Journal of Pure and Applied Algebra}, 228(10):107693, 2024.

\bibitem[CD14]{cziszter-domokos}
K{\'a}lm{\'a}n Cziszter and M{\'a}ty{\'a}s Domokos.
\newblock Groups with large {N}oether bound.
\newblock {\em Annales de l'Institut Fourier}, 64(3):909--944, 2014.

\bibitem[DH00]{domokos-hegedus}
M{\'a}ty{\'a}s Domokos and P{\'a}l Heged{\H{u}}s.
\newblock Noether's bound for polynomial invariants of finite groups.
\newblock {\em Archiv der Mathematik}, 74(3):161--167, 2000.

\bibitem[DM17]{derksen2017polynomial}
Harm Derksen and Visu Makam.
\newblock Polynomial degree bounds for matrix semi-invariants.
\newblock {\em Advances in Mathematics}, 310:44--63, 2017.

\bibitem[FKMP21]{ferraro2021noether}
Luigi Ferraro, Ellen Kirkman, W~Moore, and Kewen Peng.
\newblock On the {N}oether bound for noncommutative rings.
\newblock {\em Proceedings of the American Mathematical Society},
  149(7):2711--2725, 2021.

\bibitem[FKW07]{fleischmann2007homomorphisms}
Peter Fleischmann, Gregor Kemper, and Chris Woodcock.
\newblock Homomorphisms, localizations and a new algorithm to construct
  invariant rings of finite groups.
\newblock {\em Journal of Algebra}, 309(2):497--517, 2007.

\bibitem[Fle00]{fleischmann2000noether}
Peter Fleischmann.
\newblock The {N}oether bound in invariant theory of finite groups.
\newblock {\em Advances in Mathematics}, 156(1):23--32, 2000.

\bibitem[Fog01]{fogarty2001noether}
John Fogarty.
\newblock On {N}oether's bound for polynomial invariants of a finite group.
\newblock {\em Electronic Research Announcements of the American Mathematical
  Society}, 7(2):5--7, 2001.

\bibitem[FSSW06]{fleischmann2006noethermodular}
Peter Fleischmann, M~Sezer, R~James Shank, and Chris~F Woodcock.
\newblock The {N}oether numbers for cyclic groups of prime order.
\newblock {\em Advances in mathematics}, 207(1):149--155, 2006.

\bibitem[Gan19]{gandini2019ideals}
Francesca Gandini.
\newblock {\em Ideals of subspace arrangements}.
\newblock PhD thesis, University of Michigan, 2019.

\bibitem[HL16]{hubert-labahn}
Evelyne Hubert and George Labahn.
\newblock Computation of invariants of finite abelian groups.
\newblock {\em Mathematics of Computation}, 85(302):3029--3050, 2016.

\bibitem[LG87]{lekkerkerker-gruber}
Cornelis~Gerrit Lekkerkerker and Peter~M Gruber.
\newblock {\em Geometry of numbers}.
\newblock Elsevier, second edition, 1987.

\bibitem[Noe15]{noether}
Emmy Noether.
\newblock Der endlichkeitssatz der invarianten endlicher gruppen.
\newblock {\em Mathematische Annalen}, 77(1):89--92, 1915.

\bibitem[PV94]{popov-vinberg}
Vladimir~L Popov and Ernest~B Vinberg.
\newblock Invariant theory.
\newblock In {\em Algebraic geometry IV}, pages 123--278. Springer, 1994.

\bibitem[Sch91]{schmid1991finite}
Barbara~J Schmid.
\newblock Finite groups and invariant theory.
\newblock In {\em Topics in invariant theory}, pages 35--66. Springer, 1991.

\bibitem[Sez02]{sezer2002sharpening}
M{\"u}fit Sezer.
\newblock Sharpening the generalized {N}oether bound in the invariant theory of
  finite groups.
\newblock {\em Journal of Algebra}, 254(2):252--263, 2002.

\bibitem[Sym11]{symonds2011castelnuovo}
Peter Symonds.
\newblock On the {C}astelnuovo-{M}umford regularity of rings of polynomial
  invariants.
\newblock {\em Annals of mathematics}, pages 499--517, 2011.

\end{thebibliography}

\end{document}